\documentclass[oneside,english]{amsart}
\usepackage[T1]{fontenc}
\usepackage[latin9]{inputenc}
\usepackage{amstext}
\usepackage{amsthm}
\usepackage{amssymb}
\usepackage{setspace}
\usepackage{amsrefs}
\makeatletter
\numberwithin{equation}{section}
\numberwithin{figure}{section}
\theoremstyle{plain}
\newtheorem{thm}{\protect\theoremname}
  \theoremstyle{plain}
  \newtheorem{lem}[thm]{\protect\lemmaname}
  \theoremstyle{plain}
  \newtheorem{prop}[thm]{\protect\propositionname}
  \theoremstyle{plain}
  \newtheorem{cor}[thm]{\protect\corollaryname}
  \theoremstyle{plain}
  \newtheorem*{thm*}{\protect\theoremname}
  \theoremstyle{remark}
  \newtheorem{rem}[thm]{\protect\remarkname}

\usepackage{multicol}

\makeatother

\usepackage{babel}
  \providecommand{\corollaryname}{Corollary}
  \providecommand{\lemmaname}{Lemma}
  \providecommand{\propositionname}{Proposition}
  \providecommand{\remarkname}{Remark}
  \providecommand{\theoremname}{Theorem}
\providecommand{\theoremname}{Theorem}

\title{A note on depth preservation}

\author{Manish Mishra}
\address{
Department of Mathematics \\
Indian Institute for Science Education and Research \\
Dr. Homi Bhabha Road, Pashan \\
Pune 411 008 \\
India 
}
\email{manish@iiserpune.ac.in}
\author{Basudev Pattanayak}
\email{basudev.pattanayak@students.iiserpune.ac.in}

\begin{document}

\begin{abstract}
We show that for a wildly ramified torus, depth is not preserved in
general under local Langlands correspondence for tori.
\end{abstract}

\maketitle

\section{Introduction}

Let $K$ be a non-archimedean local field and let $W_{K}$ denote
its Weil group. Local class field theory (LCFT) tells us that there
is a canonical isomorphism $K^{\times}\cong W_{K}^{\mathrm{ab}}$
and this isomorphism respects the numbering on the filtration subgroups
$\{K_{r}^{\times}\}_{r\geq0}$ of $K^{\times}$ and the upper numbering
on the filtration subgroups $\{W_{K}^{r}\}_{r\geq0}$ of $W_{K}^{\mathrm{}}$.
Local Langlands Correspondence (LLC) stipulates a vast generalization
of the LCFT isomorphism. In LLC, irreducible representations $\mathrm{Irr}(G(K))$
of the $K$-points of a reductive $K$-group $G$ are expected to
be parametrized by arithmetic objects called Langlands parameters
$\Phi(G)$ in a certain natural way. For each $\pi\in\mathrm{Irr}(G(K))$,
Moy-Prasad theory associates an invariant called depth $\mathrm{dep}(\pi)$.
Also for each $\phi\in\Phi(G)$, one defines the notion of depth $\mathrm{dep}(\phi)$.
It is the smallest number such that $\phi$ is trivial on $W_{K}^{r}$
for all $r>\mathrm{dep}(\phi)$. If $\phi$ associates to $\pi$ under
LLC, then one expects that quite fairly $\mathrm{dep}(\pi)=\mathrm{dep}(\phi)$.
This is known in many cases (see the introduction in \cite{ABPS2016}
for a survey). However, counter examples have been constructed for
inner forms of $\mathrm{SL}_{n}(K)$ \cite{ABPS2016} and in the case
of $\mathrm{SL}_{2}(K)$ when $K$ has characteristic $2$ \cite{AMPS2017}. 

Now let $T=\mathrm{R}_{K^{\prime}/K}\mathbb{G}_{\mathrm{m}}$ where
$K^{\prime}$ is a finite separable extension of $K$ and $R_{K^{\prime}/K}$
denotes the Weil restriction and let $\lambda_{T}: \chi\in\mathrm{Irr}(T(K))\mapsto \lambda_{T}(\chi) \in\Phi(T)$
under LLC. In this note, we show that $\varphi_{K^{\prime}/K}(e\cdot\mathrm{dep}(\chi))=\mathrm{dep}(\lambda_{T}(\chi))$
where $\varphi_{K^{\prime}/K}$ is the Hasse-Herbrand function and
$e$ is the ramification index of $K^{\prime}/K$. Thus for all postitive
depth characters $\chi$, $\mathrm{dep}(\lambda_{T}(\chi))>\mathrm{dep}(\chi)$.
When $T$ is a tamely induced wildly ramified torus (see Sec. \ref{subsec:tame induced}), we show that $T(K)$ admits
characters for which depth is not preserved under LLC. In Section \ref{sec:Example},
we compute Hasse-Herbrand function for a certain wildly ramified extension
of a cyclotomic field to illustrate the failure of depth preservation. 

The proofs in Section \ref{sec:depthchangeLLC} follow closely the proofs in \cite{Yu} and \cite{MM2015}.
\section{Review of ramification groups}

Let $K$ be a non-archimedean local field and let $L$ be a finite
Galois extension of $K$. Write $\mathcal{O}_{L}$, $\mathfrak{p}_{L}$
for the ring of integers of $L$ and the maximal ideal of $\mathcal{O}_{L}$.
For $i\geq-1$, define $G_{i}$ to be the set of all $s\in G:=\mathrm{Gal}(L/K)$
such that $s$ operates trivially on $\mathcal{O}_{L}/\mathfrak{p}_{L}^{i+1}$.
Then $G_{-1}=G$. The groups $G_{i}$ are called ramification groups.
They form a decreasing filtration of normal subgroups. Extend the
definition of $G_{u}$ for all real numbers $u\geq-1$ by setting
\[
G_{u}=G_{i}\text{ where }i\text{ is the least integer }\geq u.
\]

This numbering of ramification groups is called \textit{lower numbering}.
Lower numbering behaves well with respect to intersections, i.e.,
if $H$ is a subgroup of $G$, then $G_{u}\cap H=H_{u}$. 

\subsubsection*{Upper numbering of ramification groups}

Define $\varphi_{L/K}:[-1,\infty)\rightarrow\mathbb{R}$ to be the
map $r\mapsto\int_{0}^{r}\frac{1}{(G_{0}:G_{t})}dt$ where $(G_{0}:G_{u}):=(G_{u}:G_{0})^{-1}$
for $u\in[-1,0)$. The function $\varphi_{L/K}$ is called the Hasse-Herbrand
function. It has the basic properties \cite{Serre}:

\begin{itemize}
\item[(a)]

$\varphi_{L/K}$ is continuous, piecewise linear, increasing and concave.

\item[(b)]$\varphi_{L/K}(0)=0$.

\item[(c)]$\varphi_{L/K}$ is a homeomorphism of $[-1,\infty)$ onto
itself. 

\item[(d)]If $H$ is a normal subgroup of $G$, then $\varphi_{L/K}=\varphi_{L^{H}/K}\circ\varphi_{L/L^{H}}$.

\end{itemize}

If an extension $M/K$ is not Galois, define $\varphi_{M/K}=\varphi_{M^{\prime}/K}\circ\varphi_{M/M^{\prime}}^{-1}$,
where $M^{\prime}$ is a Galois extension of $K$ containing $M$.
The inverse $\varphi_{M/K}^{-1}$ is denoted $\psi_{M/K}$.

Define an \textit{upper numbering }on ramification groups by setting
$G^{v}=G_{u}$ if $v=\varphi_{L/K}(u)$. Upper numbering behaves well
with respect to quotients, i.e., if $H$ is a normal subgroup of $G$,
then
\[
(G/H)^{v}=G^{v}H/H.
\]

For an infinite Galois extension $\Omega$ of $K$, define the ramification
groups on $G=\mathrm{Gal}(\Omega/K)$ by:
\[
G^{v}=\underset{F/K\text{ finite}}{\underleftarrow{\mathrm{lim}}}\mathrm{Gal}(F/K)^{v}.
\]

Now let $L/K$ be Galois extension of local fields and let $F$ be
a finite extension of $K$ contained in $L$. Write $G=\mathrm{Gal}(L/K)$
and $H=\mathrm{Gal}(L/F)$. 
\begin{lem}
\label{lem:psi}For all $r\geq0$, $G^{r}\cap H=H^{\psi_{F/K}(r)}$.
\end{lem}
\begin{proof}
Let $E$ be a fintie Galois extension of $K$ in $L$ containing $F$.
Write $I=\mathrm{Gal}(L/E)$. Then 
\begin{eqnarray*}
(G/I)^{r}\cap(H/I) & = & (G/I)_{\psi_{E/K}(r)}\cap(H/I)\\
 & = & (H/I)_{\psi_{E/K}(r)}\\
 & = & (H/I)^{\varphi_{E/F}(\psi_{E/K}(r))}\\
 & = & (H/I)^{\varphi_{E/F}\psi_{E/F}\psi_{F/K}(r)}\\
 & = & (H/I)^{\psi_{F/K}(r)}.
\end{eqnarray*}
 The lemma now follows by taking inverse limit over $E$.
\end{proof}

\section{Notion of depth}

Let $G=\mathrm{Gal}(L/K)$ where $L$ and $K$ are local fields and
let $M$ be a $G$-module. Define the depth of $M$ to be: 
\[
\mathrm{dep}_{G}(M)=\mathrm{inf}\{r\geq0\mid M^{G^{s}}\neq0\text{ }\forall\text{ }s>r\}.
\]

Define the depth of a co-cycle $\varphi\in\mathrm{H}^{1}(G,M)$ to
be:
\[
\mathrm{dep}_{G}(\varphi)=\mathrm{inf}\{r\geq0\mid G^{s}\subset\mathrm{ker}(\varphi)\text{ }\forall\text{}s>r\}.
\]

\section{Depth change under induction}

Let $G=\mathrm{Gal}(L/K)$ and $H=\mathrm{Gal}(L/F)$ where $K\subseteq F\subseteq L$
and $F/K$ finite Galois. Let $N$ be an $H$-module. 
\begin{prop}
$\mathrm{dep}_{H}(N)=\psi_{F/K}(\mathrm{dep}_{G}(\mathrm{Ind}_{H}^{G}N)).$
\end{prop}
\begin{proof}
By Mackey theory, 
\[
\mathrm{Res}_{G^{r}}(\mathrm{Ind}_{H}^{G}N)=\bigoplus_{g\in G^{r}\backslash G\//H}\mathrm{Ind}_{G^{r}\cap H}^{G^{r}}N^{g}.
\]
 Here $\mathrm{Res}$ denotes the restriction functor and $N^{g}$
denotes the $g$-twisted module $N$. By Lemma 1, $G^{r}\cap H=H^{\psi_{F/K}(r)}$.
Thus 
\begin{eqnarray*}
(\mathrm{Ind}_{H}^{G}N)^{G^{r}}\neq0 & \iff & (\mathrm{Ind}_{H^{\psi_{F/K}(r)}}^{G^{r}}N^{g})^{G^{r}}\neq0\text{ for some }g\in G^{r}\backslash G\//H\\
 & \iff & (N^{g})^{H^{\psi_{F/K}(r)}}\neq0\\
 & \iff & (N)^{H^{\psi_{F/K}(r)}}\neq0.
\end{eqnarray*}
\end{proof}

\section{Depth change under Shapiro's isomorphism}

Again $G=\mathrm{Gal}(L/K)$ and $H=\mathrm{Gal}(L/F)$ where $K\subseteq F\subseteq L$
and $F/K$ is any finite extension and let $N$ be an $H$-module. 

Shapiro's lemma states that the map
\[
\mathrm{Sh}:\mathrm{H}^{1}(G,\mathrm{Ind}_{H}^{G}N)\rightarrow\mathrm{H}^{1}(H,N)
\]
defined by
\[
\gamma\mapsto(h\mapsto\gamma(h)(1))
\]
is an isomorphism. We wish to relate the depth of co-cycles under
this isomorphism. We first observe the following:
\begin{lem}
\label{lem:ABC}Let A be a group, $B$ and $C$ subgroups of $A$
with $C$ being normal in $A$. Let $M$ be a $B$-module. Then there
is a canonical isomorphism of $A/C$-modules:
\[
(\mathrm{Ind}_{B}^{A}M)^{C}\cong\mathrm{Ind}_{B/B\cap C}^{A/C}M{}^{C\cap B}.
\]
\end{lem}
\begin{proof}
The map $f\in(\mathrm{Ind}_{B}^{A}M)^{C}\mapsto\tilde{f}:=(gC\mapsto f(g))\in(\mathrm{Ind}_{B/B\cap C}^{A/C}M{}^{C\cap B})$
is easily verified to be the required isomorphism. 
\end{proof}
\begin{lem}
\label{lem:sh-1}For $r\geq0$, Shapiro's lemma induces an isomorphism
$\mathrm{H^{1}}(G/G^{r},(\mathrm{Ind}_{H}^{G}N)^{G^{r}})\cong\mathrm{H^{1}}(H/H^{\psi_{F/K}(r)},N^{H^{\psi_{F/K}(r)}}).$
\end{lem}
\begin{proof}
We have 
\begin{eqnarray*}
\mathrm{H^{1}}(G/G^{r},(\mathrm{Ind}_{H}^{G}N)^{G^{r}}) & \cong & \mathrm{H^{1}}(G/G^{r},\mathrm{Ind}_{H/G^{r}\cap H}^{G/G^{r}}N{}^{G^{r}\cap H})\\
 & \cong & \mathrm{H^{1}}(G/G^{r},\mathrm{Ind}_{H/H^{\psi_{F/K}(r)}}^{G/G^{r}}N^{H^{\psi_{F/K}(r)}})\\
 & \cong & \mathrm{H^{1}}(H/H^{\psi_{F/K}(r)},N^{H^{\psi_{F/K}(r)}}).
\end{eqnarray*}

The first isomorphism follows from Lemma \ref{lem:ABC}, second from
Lemma \ref{lem:psi} and the last from Shapiro's lemma. 
\end{proof}
Write $\mathrm{H^{1}}(G,\mathrm{Ind}_{H}^{G}N)_{\mathrm{adm}}=\bigcup_{r\geq0}\mathrm{H^{1}}(G/G^{r},(\mathrm{Ind}_{H}^{G}N)^{G^{r}}).$ 
\begin{cor}
\label{cor:shapiro}
If $\lambda\in\mathrm{H^{1}}(G,\mathrm{Ind}_{H}^{G}N)_{\mathrm{adm}}$
, then $\mathrm{dep}_{G}(\lambda)=\varphi_{F/K}(\mathrm{dep}_{H}(\mathrm{Sh}(\lambda))$. 
\end{cor}
\begin{proof}

Let $\mathrm{dep}_{G}(\lambda)= r$. Then $G^s \subset \mathrm{ker}(\lambda)$ if $s > r$. By Lemma \ref{lem:sh-1}, this implies  $H^{\psi_{F/K}(s)} \subset \mathrm{ker}(\mathrm{Sh}(\lambda))$ if $s > r$. Therefore $\mathrm{dep}_{H}(\mathrm{Sh}(\lambda))\leq \psi_{F/K}(\mathrm{dep}_{G}(\lambda))$. The argument is reversible showing that $\mathrm{dep}_{H}(\mathrm{Sh}(\lambda))\geq \psi_{F/K}(\mathrm{dep}_{G}(\lambda))$. Therefore $\mathrm{dep}_{H}(\mathrm{Sh}(\lambda))= \psi_{F/K}(\mathrm{dep}_{G}(\lambda))$.

%Consider $\lambda\in\mathrm{H^{1}}(G,\mathrm{Ind}_{H}^{G}N)_{\mathrm{adm}}$. If $\mathrm{dep}_{G}(\lambda)= r$, then $G^s \subset \mathrm{ker}(\lambda)$ if and only if $s > r$, which is equivalent to $H^{\psi_{F/K}(s)} \subset \mathrm{ker}(\mathrm{Sh}(\lambda))$ if and only if $s > r$, by using Lemma \ref{lem:sh-1}. Therefore $\mathrm{dep}_{G}(\lambda)= r = \varphi_{F/K}(\mathrm{dep}_{H}(\mathrm{Sh}(\lambda))$.

\end{proof}

\section{\label{sec:LLC}Langlands correspondence for tori}

We review here the statement of local Langlands correspondence for
tori as stated and proved in \cite{Yu}.

\subsection{Special case \label{induced torus}}

Let $T=\mathrm{R}_{K^{\prime}/K}\mathbb{G}_{m}$ where $K^{\prime}$
is a finite separable extension of $K$ and $R_{K^{\prime}/K}$ denotes
the Weil restriction. Then $T(K)=K^{\prime\times}$ and the group
of characters $X^{*}(T)$ is canonically a free $\mathbb{Z}$-module
with basis $W_{K}/W_{K^{\prime}}$ where $W_{K}$ (resp. $W_{K^{\prime}}$)
denotes the Weil group of $K$ (resp. $K^{\prime}$). From this, it
follows that the complex dual $\hat{T}$ of $T$ is canonically isomorphic
to $\mathrm{Ind}_{W_{K^{\prime}}}^{W_{K}}\mathbb{C}^{\times}$. We
get, 
\begin{eqnarray}
\mathrm{Hom}(T(K),\mathbb{C^{\times}}) & \cong & \mathrm{Hom}(K^{\prime\times},\mathbb{C}^{\times})\nonumber \\
 & \cong & \mathrm{Hom}(W_{K^{\prime}},\mathbb{C}^{\times})\label{class field}\\
 & \cong & \mathrm{H}^{1}(W_{K^{\prime}},\mathbb{C}^{\times})\nonumber \\
 & \cong & \mathrm{H}^{1}(W_{K},\mathrm{Ind}_{W_{K^{\prime}}}^{W_{K}}\mathbb{C}^{\times})\label{shapiro}\\
 & \cong & \mathrm{H}^{1}(W_{K},\hat{T)}.\nonumber 
\end{eqnarray}
 The isomorphism \ref{class field} follows by class field theory
and the isomorphism \ref{shapiro} by Shapiro's lemma.

\subsection{The LLC for tori in general }
\begin{thm*}
\cite{L97}There is a unique family of homomorphisms 
\[
\lambda_{T}:\mathrm{Hom}(T(K),\mathbb{C}^{\times})\rightarrow\mathrm{H}^{1}(W_{K},\hat{T})
\]
with the following properties:

\begin{enumerate}
\item $\lambda_{T}$ is additive functorial in $T$, i.e., it is a morphism
between two additive functors from the category of tori over $K$
to the category of abelian groups;
\item For $T=R_{K^{\prime}/K}\mathbb{G}_{m}$, where $K^{\prime}/K$ is
a finite separable extension, $\lambda_{T}$ is the isomorphism described
in Section \ref{induced torus}. 
\end{enumerate}
\end{thm*}

\section{Depth change for tori under LLC}
\label{sec:depthchangeLLC}
We keep the notations as in Section \ref{sec:LLC}. Let $M$ be a
local field. Recall that $M^{\times}$ admits a filtration $\{M_{r}^{\times}\}_{r\geq0}$
where $M_{0}^{\times}$ is the units of the ring of integers and for
$r>0$, $M_{r}^{\times}:=\{x\in M\mid\mathrm{ord}_{M}(x-1)\geq r\}$.
Here $\mathrm{ord}_{M}$ is the valuation of $M$ normalised so that
$\mathrm{ord}_{M}(M^{\times})=\mathbb{Z}$. Under local class field
theory isomorphism
\[
M_{r}^{\times}\cong(W_{M}^{r})^{\mathrm{ab}}.
\]

We recall that $T(K)$ carries a Moy-Prasad filtration $\{T(K)_{r}\}_{r\geq0}$
\cite{MP96}. The depth $\mathrm{dep}_{T}(\chi)$ of a character $\chi:T(K)\rightarrow\mathbb{C}^{\times}$
is defined to be 
\[
\mathrm{inf}\{r\geq0\mid T(K)_{s}\subset\mathrm{ker}(\chi)\text{ for }s>r\}.
\]

The group $T(K)_{0}$ is called the Iwahori subgroup of $T(K)$. It
is a subgroup of finite index in the maximal compact subgroup of $T(K)$.
When $T=\mathrm{R}_{K^{\prime}/K}\mathbb{G}_{\mathrm{m}}$, then for
$r>0$, 
\begin{eqnarray}
T(K)_{r} &=& \{x\in T(K)=K^{\prime\times}\mid\mathrm{ord}_{K}(x-1)\geq r\} \label{val_filt}\\
         &=& \{x\in K^{\prime\times}\mid\mathrm{ord}_{K^{\prime}}(x-1)\geq er\}\label{norm_val_ext}\\
         &=& K_{er}^{\prime\times}. \nonumber
\end{eqnarray}
%$T(K)_{r}=\{x\in T(K)=K^{\prime\times}\mid\mathrm{ord}_{K}(x-1)\geq r\}=K_{er}^{\prime\times}$
%\cite[Sec. 4.2]{Yu03}. 
Here $\mathrm{ord}_{K}$ is the valuation
on $K^{\prime}$ normalised so that $\mathrm{ord}_{K}(K^{\times})=\mathbb{Z}$
and $e$ is the ramification index of $K^{\prime}/K$. The equality \ref{val_filt} follows from \cite[Sec. 4.2]{Yu03} and the equality \ref{norm_val_ext} follows from the fact that $\mathrm{ord}_{K^{\prime}}(\alpha) = e \cdot \mathrm{ord}_{K}(\alpha)$ for all $\alpha \in K^{\times}.$ 
\begin{thm}
\label{thm:main}Let $T=\mathrm{R}_{K^{\prime}/K}\mathbb{G}_{\mathrm{m}}$,
where $K^{\prime}/K$ is a finite separable extension of local fields
of ramification index $e$. Then for $r\geq0$, the local Langlands
correspondence for tori induces an isomorphism:
\[
\mathrm{Hom}(T(K)/T(K)_{r},\mathbb{C}^{\times})\cong\mathrm{H}^{1}(W_{K}/W_{K}^{\varphi_{K^{\prime}/K}(er)},\hat{T}^{W_{K}^{\varphi_{K^{\prime}/K}(er)}}).
\]
\end{thm}
\begin{proof}
The case $r=0$ is a special case of \cite[Theorem 7]{MM2015}.
For $r>0$, this follows by 
\begin{eqnarray}
\mathrm{Hom}(T(K)/T(K)_{r},\mathbb{C^{\times}}) & \cong & \mathrm{Hom}(K^{\prime\times}/K_{er}^{\prime\times},\mathbb{C}^{\times})\nonumber \\
 & \cong & \mathrm{Hom}(W_{K^{\prime}}/W_{K^{\prime}}^{er},\mathbb{C}^{\times})\label{class field-1}\\
 & \cong & \mathrm{H}^{1}(W_{K^{\prime}}/W_{K^{\prime}}^{er},\mathbb{C}^{\times})\nonumber \\
 & \cong & \mathrm{H}^{1}(W_{K}/W_{K}^{\varphi_{K^{\prime}/K}(er)},(\mathrm{Ind}_{W_{K^{\prime}}}^{W_{K}}\mathbb{C}^{\times})^{W_{K}^{\varphi_{K^{\prime}/K}(er)}})\label{shapiro-1}\\
 & \cong & \mathrm{H}^{1}(W_{K}/W_{K}^{\varphi_{K^{\prime}/K}(er)},(\hat{T)}^{W_{K}^{\varphi_{K^{\prime}/K}(er)}}).\nonumber 
\end{eqnarray}

Here, the isomorphism \ref{shapiro-1} follows from Lemma \ref{lem:sh-1}. 
\end{proof}
\begin{cor}
\label{cor:T}For $T$ as in Theorem \ref{thm:main}
\[
\varphi_{K^{\prime}/K}(e\cdot\mathrm{dep}_{T}(\chi))=\mathrm{dep}_{W_{K}}(\lambda_{T}(\chi)).
\]
\end{cor}
\begin{proof}
This follows from an argument analogous to the argument in the proof of the Corollary \ref{cor:shapiro}.
\end{proof}

\begin{rem}
\label{remark}
The slope of the map $r\mapsto\varphi_{K^{\prime}/K}(er)$ at a differentiable
point $r$ is $\frac{e}{(G_{0}:G_{r})}\geq1$. Thus, when $K^{\prime}/K$
is a wildly ramified extension, $\varphi_{K^{\prime}/K}(er)>r$ and
consequently $\mathrm{dep}_{T}(\chi)<\mathrm{dep}_{W_{K}}(\lambda_{T}(\chi))$.

When $K^{\prime}/K$ is a tamely ramified extension, $\varphi_{K^{\prime}/K}(r)=\frac{r}{e}$.
Therefore in this case, Corollary \ref{cor:T} simplifies to, 
\[
\mathrm{dep}_{T}(\chi)=\mathrm{dep}_{W_{K}}(\lambda_{T}(\chi)).
\]
 This is a special case of Depth-preservation Theorem of Yu for tamely
ramified tori \cite[Sec. 7.10]{Yu}. 
\end{rem}

\subsection{\label{subsec:tame induced}Case of a tamely induced tori}

Recall that a $K$-torus is called \textit{induced} if it is of the
form $\Pi_{i=1}^{k}R_{L_{i}/K}\mathbb{G}_{\mathrm{m}}$, where $L_{i}$
are finite separable extensions of $K$. A $K$-torus $\mathcal{T}$
is called \textit{tamely induced} if $\mathcal{T}\otimes_{K}K_{t}$
is an induced torus for some tamely ramified extension $K_{t}$ of
$K$. In this section, we compare depths under LLC for such tori following
the proof in \cite[Sec. 7.10]{Yu}.

Let $T$ be a tamely induced $K$-torus. Then there exists an induced
torus $T^{\prime}=\prod_{i=1}^{n}\mathrm{R}_{K_{i}^{\prime}/K}\mathbb{G}_{\mathrm{m}}$
such that $T^{\prime}\twoheadrightarrow T$ and $C_0:=\mathrm{ker}(T^{\prime}\rightarrow T)$
is connected. Further $T^{\prime}(K)_{r}\twoheadrightarrow T(K)_{r}$
$\forall r>0$ (see proof in \cite[Lemma 4.7.4]{Yu03}). Let $\chi\in\mathrm{Hom}(T(K),\mathbb{C}^{\times})$
and let $\chi^{\prime}$ denote its lift to $T^{\prime}(K)$. Then
\begin{equation}
\mathrm{dep}_{T}(\chi)=\mathrm{dep}_{T^{\prime}}(\chi^{\prime})=\mathrm{sup}\{\mathrm{dep}_{T_{i}^{\prime}}(\chi_{i}^{\prime})\mid1\leq i\leq n\}.
\end{equation}
 Here $T_{i}^{\prime}$ denotes $\mathrm{R}_{K_{i}^{\prime}/K}\mathbb{G}_{\mathrm{m}}$
and $\chi_{i}^\prime=\chi^\prime|_{T_{i}^{\prime}}$. By functoriality, $\lambda_{T^{\prime}}(\chi^{\prime})$
is the image of $\lambda_{T}(\chi)$ under $\mathrm{H}^{1}(W_{K},\hat{T})\rightarrow\mathrm{H}^{1}(W_{K},\hat{T^{\prime}})$
and therefore $\mathrm{dep}_{W_{K}}(\lambda_{T}(\chi))=\mathrm{dep}_{W_{K}}(\lambda_{T^{\prime}}(\chi^{\prime}))$.
But 
\begin{eqnarray*}
\mathrm{dep}_{W_{K}}(\lambda_{T^{\prime}}(\chi^{\prime})) & = & \mathrm{sup}\{\mathrm{dep}_{W_{K}}(\lambda_{T_{i}^{\prime}}(\chi_{i}^{\prime}))\mid1\leq i\leq n\}\\
 & = & \mathrm{sup}\{\varphi_{K_{i}^{\prime}/K_{i}}(e_{i}\cdot\mathrm{dep}_{T_{i}^{\prime}}(\chi_{i}^{\prime}))\mid1\leq i\leq n\}.\\
 & \geq & \mathrm{sup}\{\mathrm{dep}_{T_{i}^{\prime}}(\chi_{i}^{\prime})\mid1\leq i\leq n\}.
\end{eqnarray*}

Here $e_{i}$ denotes the ramification index of $K_{i}^{\prime}/K$.
Thus 
\begin{equation}
\label{eq:ineqality}
\mathrm{dep}_{W_{K}}(\lambda_{T}(\chi))\geq\mathrm{dep}_{T}(\chi).
\end{equation}
Now assume $T$ is wildly ramified. We will now produce a character of $T(K)$ for which  the  inequality (\ref{eq:ineqality}) is strict. We can assume without loss of generality that $T_0:=\mathrm{R}_{K_{1}^{\prime}/K}\mathbb{G}_\mathrm{m}$ is wildly ramified. Let $\chi_{0}^{\prime}$ be a positive depth character of $T_0(K)$ which is trivial on $C_0\cap T_0(K)$. Extend $\chi_{0}^\prime$ trivially to a character $\chi^\prime$ of $T^\prime(K)$. Then since $\mathbb{C}^\times$ is divisible, the character $\chi^\prime$ lifts to a character $\chi$ of $T(K)$. By Remark \ref{remark}, 
$\mathrm{dep}_{T_0}(\chi_{0}^\prime)<\mathrm{dep}_{W_{K}}(\lambda_{T_0}(\chi_{0}^\prime))$. Since $\mathrm{dep}_{T}(\chi)=\mathrm{dep}_{T_0}(\chi_{0}^\prime)$ and $\mathrm{dep}_{W_{K}}(\lambda_{T}(\chi))=\mathrm{dep}_{W_{K}}(\lambda_{T_0}(\chi_{0}^\prime))$, it follows that  the inequality (\ref{eq:ineqality}) is strict for this choice of $\chi$. 

\section{\label{sec:Example}An Example}

Let $K=\mathbb{Q}_{p}$, $L=K(\zeta_{p^{n}})$, where $\zeta_{p^{n}}$
denotes a primitive $p^{n}$th root of unity, $n\geq1$. Then $L/K$
is a totally ramified extension of degree $(p-1)p^{n-1}.$ Consider
the intermediate extension $F=K(\zeta_{p})$ of $K$ of degree $p-1$
over $K$. Then $L/F$ is a wildly ramified extension. Write $G=\mathrm{Gal}(L/K)$
and $H=\mathrm{Gal}(L/F)$. 
\begin{lem}
For $1\leq r$ and $\varphi_{L/F}(r)=(p-1)\varphi_{L/K}(r)$.
\end{lem}
\begin{proof}
We first note that since we considering abelian extensions, the jumps
in filtration occur at integer values. We have for $r\geq1$,
\begin{eqnarray*}
\varphi_{L/K}(r) & = & \int_{0}^{r}\frac{dt}{(G_{0}:G_{t})}\\
 & = & \int_{0}^{1}\frac{dt}{(G_{0}:G_{t})}+\int_{1}^{r}\frac{dt}{(G_{0}:G_{t})}\\
 & = & \frac{1}{p-1}+\int_{1}^{r}\frac{dt}{(G_{0}:G_{t})}\\
 & = & \frac{1}{p-1}+\int_{1}^{r}\frac{(H_{0}:H_{t})}{(G_{0}:G_{t})}\frac{dt}{(H_{0}:H_{t})}\\
 & = & \frac{1}{p-1}+\int_{1}^{r}\frac{(G_{t}:H_{t})}{(G_{0}:H_{0})}\frac{dt}{(H_{0}:H_{t})}\\
 & = & \frac{1}{p-1}+\frac{1}{p-1}\int_{1}^{r}\frac{dt}{(H_{0}:H_{t})}.
\end{eqnarray*}
 The last equality holds because $G_{t}=H_{t}$ for $t\geq1$ and
$(G_{0}:H_{0})=p-1$. Thus 
\begin{eqnarray*}
\varphi_{L/K}(r) & = & \frac{1}{p-1}+\frac{1}{p-1}(\varphi_{L/F}(r)-\int_{0}^{1}\frac{1}{(H_{0}:H_{t})}dt)\\
 & = & \frac{1}{p-1}+\frac{1}{p-1}(\varphi_{L/F}(r)-1)\\
 & = & \frac{\varphi_{L/F}(r)}{(p-1)}.
\end{eqnarray*}
\end{proof}
Write $m=p^{n}$ and let $G(m)=(\mathbb{Z}/m\mathbb{Z})^{\times}$.
By \cite[Chap IV, Prop. 17]{Serre}, $G=G(m)$. Define
\[
G(m)^{s}:=\{a\in G(m)\mid a\equiv1\text{ mod }p^{s}\}.
\]

Then $G(m)^{s}=\mathrm{Gal}(L/K(\zeta_{p^{s}}))$. The ramification
groups $G_{u}$ of $G$ are \cite[Chap IV, Prop. 18]{Serre}:
\begin{eqnarray*}
\text{} &  & G_{0}=G\\
\text{if }1\leq u\leq p-1 &  & G_{u}=G(m)^{1}\\
\text{if }p\leq u\leq p^{2}-1 &  & G_{u}=G(m)^{2}\\
\vdots &  & \vdots\\
\text{if }p^{n-1}\leq u &  & G_{u}=1.
\end{eqnarray*}

We now calculate $\varphi_{L/F}.$ 
\begin{prop}
\label{lem:HH_Calc}The Hasse-Herbrand function of the wildly ramified
extension $L/F$ is given by 
\begin{equation}
\varphi_{L/F}(r)=\begin{cases}
k(p-1)+\frac{r-p^{k}+1}{p^{k}} & \text{if }p^{k}-1<r\leq p^{k+1}-1\text{ with }0\leq k<n-1\\
(n-1)(p-1)+\frac{r-p^{n-1}+1}{p^{n-1}} & r>p^{n-1}-1
\end{cases}.
\end{equation}
\begin{proof}
We consider various cases:

\begin{itemize}
\item \textbf{Case $0<r\leq1$}
\begin{eqnarray*}
\varphi_{L/F}(r) & = & \int_{0}^{r}\frac{dt}{(H_{0}:H_{t})}\\
 & = & \frac{1}{(H_{0}:H_{1})}\int_{0}^{r}dr\\
 & = & r.
\end{eqnarray*}
\item \textbf{Case} $1<r\leq p-1$
\begin{eqnarray*}
\varphi_{L/K}(r) & = & \int_{0}^{r}\frac{dt}{(G_{0}:G_{t})}\\
 & = & \int_{0}^{1}\frac{dt}{(G_{0}:G_{1})}+\int_{1}^{r}\frac{dt}{(G_{0}:G_{t})}\\
 & = & \frac{1}{p-1}+\int_{1}^{r}\frac{dt}{(G_{0}:G(m)^{1})}\\
 & = & \frac{r}{p-1}.
\end{eqnarray*}
Therefore, $\varphi_{L/F}(r)=r$.
\item \textbf{Case} $p^{k}-1<r\leq p^{k+1}-1$ with $1\leq k<n-1$
\begin{eqnarray*}
\varphi_{L/K}(r) & = & \int_{0}^{r}\frac{dt}{(G_{0}:G_{t})}\\
 & = & \sum_{i=0}^{k-1}\int_{(p^{i}-1)}^{(p^{i+1}-1)}\frac{dt}{(G_{0}:G_{t})}+\int_{p^{k}-1}^{r}\frac{dt}{(G_{0}:G_{t})}\\
 & = & \int_{0}^{1}\frac{dt}{(G_{0}:G_{1})}+\int_{1}^{p-1}\frac{dt}{(G_{0}:G(m)^{1})}+\sum_{i=1}^{k-1}\int_{p^{i}-1}^{p^{i+1}-1}\frac{dt}{(G_{0}:G(m)^{i+1}}\\
 &  & +\int_{p^{k}-1}^{r}\frac{dt}{(G_{0}:G(m)^{k+1}}\\
 & = & \frac{1}{p-1}+\frac{p-2}{p-1}+\sum_{i=1}^{k-1}\frac{p^{i+1}-p^{i}}{(p-1)p^{i}}+\frac{r-p^{k}+1}{(p-1)p^{k}}\\
 & = & k+\frac{r-p^{k}+1}{(p-1)p^{k}}.
\end{eqnarray*}
Therefore, $\varphi_{L/F}(r)=k(p-1)+\frac{r-p^{k}+1}{p^{k}}$.
\item \textbf{Case} $r>p^{n-1}-1$
\begin{eqnarray*}
\varphi_{L/K}(r) & = & \int_{0}^{r}\frac{dt}{(G_{0}:G_{t})}\\
 & = & \int_{0}^{p^{n-1}-1}\frac{dt}{(G_{0}:G_{1})}+\int_{p^{n-1}-1}^{r}\frac{dt}{(G_{0}:G_{t})}\\
 & = & (n-1)+\frac{r-p^{n-1}+1}{(p-1)p^{n-1}}.
\end{eqnarray*}
Therefore, $\varphi_{L/F}(r)=(n-1)(p-1)+\frac{r-p^{n-1}+1}{p^{n-1}}.$ 
\end{itemize}
\end{proof}
\end{prop}
Now write $T=\mathrm{R}_{L/F}\mathbb{G}_{\mathrm{m}}$ and let $\lambda_{T}$
be as denoted in Sec. \ref{sec:LLC}. It then immediately follows from
Prop. \ref{lem:HH_Calc}:
\begin{lem}
$\varphi_{L/K}(p^{n-1}r)>r$ $\forall r>0.$ Consequently, for all
positive depth $\chi\in\mathrm{Hom}(T(K),\mathbb{C}^{\times})$, $\mathrm{dep}_{T}(\chi)<\mathrm{dep}_{W_{K}}(\lambda_{T}(\chi))$. 
\end{lem}

\section{Acknowledgement}

The authors would like to thank Anne-Marie Aubert for carefully going over this article and suggesting several improvements.

\end{document}